\documentclass[a4paper,12pt,reqno,oneside]{amsart}

\usepackage{setspace} 
\usepackage{typearea} 

\usepackage{amscd,amsmath,amssymb,verbatim}

\usepackage{epstopdf}
\usepackage{graphicx}\setkeys{Gin}{width=\linewidth}

\usepackage{ifthen}

\usepackage{stmaryrd}
\usepackage{cmll}
\usepackage{xr-hyper}
\usepackage{cite}
\usepackage[colorlinks,backref,
final,bookmarksnumbered,bookmarks]{hyperref}
\usepackage[backrefs]{amsrefs}
\usepackage{enumitem}
\newcommand{\arxiv}[2][]{\ifthenelse{\equal{#1}{}}
{\href{http://arxiv.org/abs/#2}{\tt arXiv:#2}}
{\href{http://arxiv.org/abs/math/#2}{\tt arXiv:math.#1/#2}}}

\usepackage[T1,T2A]{fontenc}
\usepackage[utf8]{inputenc}

\makeatletter
\renewcommand\subsection{\@startsection
{subsection}{2}{0cm} 
{-\baselineskip}     
{0.5\baselineskip}   
{\sffamily}} 
\makeatother

\theoremstyle{plain}
\newtheorem{theorem}{Theorem}[section]
\newtheorem{lemma}[theorem]{Lemma}
\newtheorem{corollary}[theorem]{Corollary}
\newtheorem{proposition}[theorem]{Proposition}
\newtheorem{problem}[theorem]{Problem}

\theoremstyle{definition}
\newtheorem{example}[theorem]{Example}
\newtheorem{definition}[theorem]{Definition}

\newtheoremstyle{remark}
{}{}{}{}{\itshape}{}{ }{\thmname{#1}\thmnumber{ \itshape #2.}}
\theoremstyle{remark}
\newtheorem{remark}[theorem]{Remark}

\def\R{\mathbb{R}} \def\Z{\mathbb{Z}} \def\C{\mathbb{C}} \def\tl{\tilde}
\def\x{\times} \def\but{\setminus} \def\emb{\hookrightarrow} 
\def\eps{\varepsilon} \def\phi{\varphi}
\def\xr#1{\xrightarrow{#1}} \def\xl#1{\xleftarrow{#1}} \renewcommand{\:}{\colon}
\DeclareMathOperator{\Deg}{deg} \DeclareMathOperator{\Aut}{Aut}

\begin{document}

\title{Transverse fundamental ~group and projected ~embeddings}
\author{Sergey A. Melikhov}
\address{Steklov Mathematical Institute of the Russian Academy of Sciences,
ul.\ Gubkina 8, Moscow, 119991 Russia}
\email{melikhov@mi.ras.ru}
\thanks{This work is supported by the Russian Science Foundation under grant 14-50-00005.}

\begin{abstract}
For a generic degree $d$ smooth map $f\:N^n\to M^n$ we introduce its ``transverse fundamental group'' $\pi(f)$,
which reduces to $\pi_1(M)$ in the case where $f$ is a covering, and in general admits a monodromy homomorphism
$\pi(f)\to S_{|d|}$; nevertheless, we show that $\pi(f)$ can be non-trivial already for rather simple degree $1$
maps $S^n\to S^n$.

We apply $\pi(f)$ to the problem of lifting $f$ to an embedding $N\emb M\x\R^2$: for such a lift to exist,
the monodromy $\pi(f)\to S_{|d|}$ must factor through the group of concordance classes of $|d|$-component
string links.
At least if $|d|<7$, this requires $\pi(f)$ to be torsion-free.
\end{abstract}

\maketitle

\section{Introduction}\label{1}

A generic $C^\infty$ map $f\:N\to M$ is called a {\it $k$-projected embedding}, or a {\it $k$-prem} if
there exists a $g\:N\to\R^k$ such that $f\x g\:N\to M\x\R^k$ is a smooth embedding (cf.\ \cite{ARS},
\cite{AMR}).

A necessary condition for a generic smooth map $f\:N^n\to M^m$ to be a $k$-prem
is the existence of a $\Z/2$-equivariant map from the double point set
$\Delta_f=\{(x,y)\in N\x N\mid f(x)=f(y),\, x\ne y\}$ (endowed with the restriction
of the factor exchanging involution on $N\x N$) to the sphere $S^{k-1}$
(endowed with the antipodal involution $x\mapsto -x$).
Indeed, given a $g\:N\to\R^k$ as above, we define $\phi\:\Delta_f\to S^{k-1}$ by
$\phi(x,y)=\frac{g(x)-g(y)}{||g(x)-g(y)||}$; clearly, $\phi(y,x)=-\phi(x,y)$.
One can show that when $4n-3m\le k$ and $m+k\ge\frac{3(n+1)}2$, this necessary condition is also
sufficient \cite{M2}.

Using this result, it is not hard to see that every generic smooth map $f$ between orientable smooth
$n$-manifolds is an $n$-prem for all even $n>2$; indeed, an equivariant map
$\Delta_f\to S^{n-1}$ exists for all even $n$ (including $n=2$) as observed in the proof of 
Theorem \ref{1.2} below.

\begin{problem}\label{1.1} Is every generic smooth map between orientable
surfaces a $2$-prem?
\end{problem}

It is known that the answer is affirmative in the following cases: for maps of any $2$-manifold
into $\R^2$ \cite{Ya}; for maps of $S^2$ into any orientable $2$-manifold
(Yamamoto--Akhmetiev \cite{M1}); and for maps $S^1\x S^1\to S^1\x S^1$ \cite{KW}.

Petersen proved that the answer is also affirmative for all regular coverings of degree $<60$
\cite{Pe}.
This is established as follows: a group of order $<60$ is
solvable; a solvable covering is a composition of abelian coverings; an abelian
covering over a compact polyhedron with free abelian $H_1$ is induced from
a covering over a torus $S^1\x\dots\x S^1$; a covering over a torus is equivalent
by a change of coordinates to a product of coverings over $S^1$; a product of
coverings over $S^1$ is a composition of coverings over the torus induced from
coverings over $S^1$; and a covering over $S^1$ is obviously a $2$-prem.

Petersen also proved that if a composition of two coverings is a $2$-prem, then each of them is 
a $2$-prem \cite{Pe}.
Let us note that the regular covering corresponding to the kernel of the monodromy homomorphism
of a covering $p$ factors through $p$.
Consequently, every covering with solvable monodromy group (hence in particular every covering of 
degree $<5$) between orientable surfaces is a $2$-prem.

Let us also note that a covering over a connected sum of tori is a 2-prem if it is
induced from a covering of the wedge of these tori; indeed, {\it any} covering induced from
a covering over a wedge of tori is a 2-prem (see Theorem \ref{2.7} below and
the subsequent remarks).

Finally, we should note that one motivation of Problem \ref{1.1} is that its affirmative solution
would yield an affirmative answer to the following

\begin{problem}\label{1.4} \cite{CF}, \cite{M1} Does every inverse limit of orientable
$2$-manifolds embed in $\R^4$?
\end{problem}

P. M. Akhmetiev proved that an inverse limit of stably parallelizable
$n$-manifolds embeds in $\R^{2n}$ for $n\ne 1,2,3,7$ \cite{Ah} (see \cite{M1}
for an explicit proof).
It is well-known that the $p$-adic solenoid, which is an inverse limit of
copies of $S^1$, does not embed in the plane.
Akhmetiev also constructed inverse limits of $3$- and $7$-dimensional
parallelizable manifolds that do not embed in $\R^6$, resp.\ $\R^{14}$
(see \cite{MS}).

\subsection{Content of the paper}

It is clear from Theorem \ref{1.2} below that Problem \ref{1.1} is a typical problem of 
four-dimensional topology%
\footnote{Keeping in mind, say, the 4-dimensional PL Poincar\'e conjecture,
the Andrews--Curtis conjecture, the problem of PL embeddability of
acyclic and contractible 2-polyhedra in $\R^4$, etc.} in that there is no lack
of potential counterexamples (such as 5-fold coverings and regular
60-fold coverings, not to mention generic approximations of various branched
coverings) but an obviuous lack of invariants/obstructions capable of detecting
actual counterexamples.

The present note develops one approach to constructing such an obstruction
in the case of generic maps other than coverings.
The ``transverse fundametal group'' $\pi(f)$ of a generic smooth map $f\:M\to N$
between manifolds of the same dimension is introduced in \S\ref{2}.
In the case where $f$ is a covering, $\pi(f)$ specializes to $\pi_1(M)$ and so gives nothing new.
On the other hand, we compute, for instance, that $\pi(f)$ contains an infinite cyclic subgroup for 
a certain fold map $f\:S^2\to S^2$, which is a generic $C^0$-approximation of the suspension
of the double covering $S^1\to S^1$.
This and other examples are studied in \S\ref{3}.

The following is a special case of Corollary \ref{2.10}.

\begin{theorem}\label{1.3} Let $f$ be a generic smooth map of degree $<7$ between
orientable $2$-manifolds.
If $\pi(f)$ contains torsion, then $f$ is not a $2$-prem.
\end{theorem}

The author does not know if $\pi(f)$ can contain torsion when $\pi_1(M)$ is torsion-free;
generally speaking, nothing seems to preclude from this.

From the viewpoint of algebraic topology, the elements of $\pi(f)$ are
analogous, or rather dual, to spherical classes in the $2$-homology of
a $4$-manifold (see Remark \ref{2.1}).
Even though the technique enabling us to show that $\pi(f)$ is well-defined
was originally developed in the course of a study of projected embeddings \cite{M1},
the present paper was written with hope that $\pi(f)$ may also find other
applications.

\subsection{A motivation: The double point obstruction}
One may look at the following straightforward obstruction to \ref{1.1}.
Let us consider, more generally, a generic smooth map $f\:N^n\to M^n$.
Take a generic lift $\bar f\:N\to M\x\R^n$ of $f$ and pick some
basepoint $b\in N$.
Each (necessarily isolated) double point $z=\bar f(x)=\bar f(y)$ of $\bar f$ has
a sign $\eps_z=\pm 1$ determined by comparing the orientations of the two
sheets of $N$ with the orientation of $M\x\R^n$.
Let us pick a path $p_x$ joining $b$ and $x$ and a path $p_y$ joining
$y$ and $b$.
Then $f(p_x p_y)$ is an $f(b)$-based loop in $M$.
The class $g_z\in G:=\pi_1(M,f(b))$ of this path is well defined up to
multiplication on both sides by elements of $H:=f_*(\pi_1(N))$.
Let $\theta(\bar f)$ be the algebraic sum
$$\sum_z \eps_zHg_zH\in\Z[H\backslash G/H]$$ of
the resulting double cosets.
If $\bar f'$ is another generic lift of $f$, a generic homotopy between
$\bar f$ and $\bar f'$ over $f$ yields an oriented bordism between the set
of double points of $\bar f$ and that of $\bar f'$.
The critical levels of this bordism consist of cancellations/introductions of
pairs $(z,z')$ such that $\eps_z=-\eps_{z'}$ and $g_z=g_{z'}$;
and (unless $f$ is a covering) of births/deaths of individual double points
$z$ such that $g_z\in H$.
Hence $\theta(f):=\theta(\bar f)=\theta(\bar f')$ is well defined.
Obviously, if $f$ is an $n$-prem, $\theta(f)=0$.

\begin{theorem}\label{1.2} If $n$ is even, $\theta(f)=0$ for every generic smooth $f\:N^n\to M^n$.
\end{theorem}

\begin{proof} Since the dimensions of $N$ and $M$ have the same parity,
$\Delta_f/T$ is orientable, where $T$ is the factor exchanging involution
on $N\x N$ (see e.g.\ \cite{M1}*{Akhmetiev's Lemma (preceding Example 5) or the beginning of \S3}).
If $\lambda$ is the line bundle associated with the double covering
$\Delta_f\to\Delta_f/T$, its Euler class $e(\lambda)$ is an element of order two 
in the cohomology group $H^1(\Delta_f/T;\,\Z_T)$ with local coefficients
(see e.g.\ \cite{M3}*{\S2}).
Since $\Delta_f/T$ is orientable, whereas the coefficients $\Z_T^{\otimes n}$ 
are constant when $n$ is even, $H^n(\Delta_f/T;\,\Z_T^{\otimes n})$ is free abelian,
and therefore $e(\lambda)^n=0$.%
\footnote{This immediately implies the existence of an equivariant map $\Delta_f\to S^{n-1}$ 
(see e.g.\ \cite{M3}*{Alternative proof of Theorem 3.2}).}

This yields an equivariant oriented null-bordism $W$ of the oriented
$0$-manifold $\Delta_{\bar f}$ in $\Delta_f$ (see e.g.\ \cite{M1}*{Lemma 7} or 
\cite{M3}*{\S3, subsections ``Geometric definition of $\vartheta(f)$'' and
``Cohomological sectional category''}).
Without loss of generality $W$ has no components without boundary.
By the definition of $\Delta_f$, we have $fpT(a)=f(y)=f(x)=fp(a)$ for each
$a=(x,y)\in\Delta_f$, where $p$ projects $N\x N$ onto the first factor.
Since $W\subset\Delta_f$, it follows that $fpT|_J=fp|_J$ for each component
$J$ of $W$.

Let $c=(x,y)$ and $d=(x',y')$ be the endpoints of a compact component $J$ 
corresponding to
double points $z=\bar f(x)=\bar f(y)$ and $z'=\bar f(x')=\bar f(y')$ of
$\bar f$, so that $\eps_z=-\eps_{z'}$.
Let $p_x$ be a path joining $b$ to $x=p(c)$ and $p_y$ a path joining
$y=pT(c)$ to $b$.
Then $p_x$ followed by $p|_J$ is a path $p_{x'}$ joining $b$ to $x'=p(d)$ and
the inverse of $pT|_J$ followed by $p_y$ is a path $p_{y'}$ joining $y'=pT(d)$ 
to $b$.
Now $fpT|_J=fp|_J$ implies that $f(p_{x'}p_{y'})$ is homotopic to
$f(p_xp_y)$, whence $g_{z'}=g_z$.

Similarly if $J$ is noncompact and so has only one endpoint $a=(x,y)$
corresponding to a double point $z=\bar f(x)=\bar f(y)$, then $g_z\in H$.
\end{proof}

The author is grateful to P. Akhmetiev and M. Yamamoto for very valuable remarks.
The paper also benefited from stimulating conversations with N. Brodskiy,
V. Chernov, J. Keesling,
E. Kudryavtseva, S. Maksymenko, R. Mikhailov and R. Sadykov.

\section{In search of non-$2$-prems}\label{2}

\subsection{Transverse fundamental group}

\begin{definition}[Pullback]
If $L$, $M$ and $N$ are smooth manifolds and $L\xr{g}M\xl{f}N$ are smooth maps,
we say that $g$ is {\it transverse} to $f$ and write $g\pitchfork f$ if
$f\x g\:N\x L\to M\x M$ is transverse to $\Delta_M$.
In this case $P:=(f\x g)^{-1}(\Delta_M)$ is a smooth submanifold of $N\x L$,
and consequently the composition of the inclusion $P\emb N\x L$ and
the projection $N\x L\to N$ is a smooth map; if additionally $f$ is generic,
then so is this composition.
This composition is called the {\it pullback} (or the ``base change map'')
of $f$ along $g$ and is denoted $g^*f$,
and its domain $P$ (also known as the {\it pullback} of the diagram $L\xr{g}M\xl{f}N$)
may be denoted $(g^*f)^{-1}(L)$.
Note that if $g$ is an embedding, then so is $f^*g$, which therefore performs
a homeomorphism between $(g^*f)^{-1}(L)=(f^*g)^{-1}(N)$ and $f^{-1}(g(L))$.
\end{definition}

\begin{definition}[Coherent homotopy]
Let $f\:N\to M$ be a generic smooth map between closed oriented connected
$n$-manifolds, $n\ge 1$, and let $b\in M$ be its regular value (in particular,
$b$ is a value of $f$, i.e.\ $b\in f(N)$.)
Consider $f$-transverse based loops $l_0,l_1\:(S^1,pt)\to (M,b)$.
A based homotopy $h\:(S^1\x I,pt\x I)\to (M,b)$ between $l_0$ and $l_1$
will be called {\it $(b,f)$-coherent} if it is $f$-transverse and every
connected component of the pullback $(h^*f)^{-1}(S^1\x I)$ that intersects
$(h^*f)^{-1}(pt\x I)$ is an annulus with one boundary component in
$(l_0^*f)^{-1}(S^1)$ and another in $(l_1^*f)^{-1}(S^1)$.
Note that some individual levels $h_t\:S^1\to M$ of a $(b,f)$-coherent homotopy
may be non-$f$-transverse, and the number of components in $(h_t^*f)^{-1}(S^1)$
may vary depending on $t$.
\end{definition}

\begin{definition}[$\pi(f)$: The case of unfolded basepoint]
Suppose first the cardinality $|f^{-1}(b)|$ equals the absolute value $|\Deg f|$
(so in particular $\Deg f\ne 0$, since we are assuming that $b\in f(N)$).
The set $\pi(f,b)$ of $b$-based $f$-transverse loops in $M$ up to $(b,f)$-coherent
homotopy is clearly a group with respect to the usual product (i.e.\
the concatenation) of loops and the usual inverse of a loop.
\end{definition}

\begin{example}
If the generic map $f\:N\to M$ is a covering, $\pi(f,b)\simeq\pi_1(M)$ since
coverings enjoy the covering homotopy property.

On the other hand, there exists, for instance, a fold map $f\:S^2\to S^2$ with
four fold curves and with $|f^{-1}(b)|=1$ such that $\pi(f)\ne 1$ (see
Examples \ref{3.2.1}, \ref{3.2.4}).
\end{example}

\begin{remark}\label{2.1} Note that the Pontryagin construction identifies $f$-transverse
framed loops in $M$ with stable maps of the mapping cylinder of $f$
into $S^{n-1}$ that are transverse to $pt\in S^{n-1}$.
Of course, homotopies between such maps, transverse to $pt$, are identified
with arbitrary $f$-transverse framed bordisms, which are not necessarily
homotopies (not to mention coherent homotopies).
Thus from the viewpoint of algebraic topology, the question of existence of
a coherent homotopy is a question of representability of a generalized
cohomology class by a genus zero cocycle extending a given representation
on the boundary.
(By a cocycle we mean a pseudo-comanifold, i.e.\ an embedded mock bundle with
codimension two singularities --- see \cite{BRS}.)
\end{remark}

\begin{definition}[$\pi(f)$: The general case]
Without loss of generality we may assume that $d:=\Deg(f)\ge 0$.
Let $j\:I\to M$ be an $f$-transverse path, and let $J=(j^*f)^{-1}(I)$.
Since $M$ is oriented, $\Delta_M$ is co-oriented in $M\x M$, hence $J$
is co-oriented in $I\x N$.
Since $I$ and $N$ are oriented, $J$ is oriented.

A component $C$ of $J$ is called a {\it positive (negative)} arc if
$j^*f|_C\:(C,\partial C)\to (I,\partial I)$ has degree $+1$ (resp.\ $-1$).
Else $C$ could be a circle or an arc with both endpoints mapping onto
the same endpoint of $I$, with $(j^*f)(C)\ne I$.
Note that the signs of the arcs reverse (along with the sign of $\Deg(f)$) when
the orientation of $M$ or $N$ is reversed; but remain unchanged when
the orientation of $I$ is reversed.
\end{definition}

\begin{lemma}\label{2.2} \cite{M1}*{\S2, proof of Observation 2}
Let $f\:N\to M$ be a generic smooth map between closed oriented connected
$n$-manifolds, $n\ge 1$, with $\Deg(f)\ge 0$.
Then there exists an $f$-transverse path $\ell\:I\to M$ such that
$(\ell^*f)^{-1}(I)$ contains no negative arcs.
\end{lemma}

Without loss of generality $a:=\ell(0)$ and $b:=\ell(1)$ are $f$-regular values.
(In fact, since any $f$-transverse path $\ell_+$ containing $\ell$ is again such
that $(\ell_+^*f)^{-1}(I)$ contains no negative arcs, $a$ and $b$
could have been any $f$-regular values given in advance.)
Let $L=(\ell^*f)^{-1}(I)$, and let $D$ be a bijection between
$[d]:=\{0,1,\dots,d-1\}$ and the set of endpoints in $(\ell^*f)^{-1}(0)$ of
the $d$ positive arcs in $L$.

Let $j\:(I,\partial I)\to(M,b)$ be any $f$-transverse loop.
Then the product $\hat j$ of the paths $\ell$, $j$ and the inverse path
$\bar\ell$ (defined by $\bar\ell(t)=\ell(1-t)$) is again such that
$\hat J:=(\hat j^*f)^{-1}(I)$
contains no negative arcs; moreover, each positive arc of $L$ is contained in
a unique positive arc of $\hat J$.

We define $\pi(f,\ell)$ to be the set of all $f$-transverse $b$-based loops
up to $b$-based homotopy $j_t$ such that the $a$-based homotopy
$\hat j_t$ is $(a,f)$-coherent.
Clearly this is a group with respect to the usual product of loops
and the usual inverse of a loop.

Furthermore, assigning to an endpoint of a positive arc in $\hat J$
the other endpoint of this arc, we get a bijection $h_{j,D}\:[d]\to [d]$.
If $[j']=[j]\in\pi(f,\ell)$, clearly $h_{j',D}=h_{j,D}$.
Hence $[j]\mapsto h_{j,D}$ defines a homomorphism
$\phi_{f,\ell,D}\:\pi(f,\ell)\to S_d$.

The following theorem says in particular that the {\it transverse fundamental
group} $\pi(f):=\pi(f,\ell)$ is well-defined up to an inner automorphism, and
its {\it monodromy map} $\phi_f:=\phi_{f,\ell,D}\:\pi(f)\to S_{|\Deg f|}$ is
well-defined up to an inner automorphism of the target group.

\begin{theorem}\label{2.3} Let $f\:N\to M$ be a generic smooth map between
closed oriented connected $n$-manifolds.
Then $\pi(f,b):=\pi(f,\ell)$ does not depend on the choice of $\ell$.
Moreover, every path $p$ joining $b'$ and $b$ induces an isomorphism
$H_p\:\pi(f,b)\to\pi(f,b')$ and a permutation $h_{p,D,D'}\in S_d$ such that
$h_{p,D,D'}\phi_{f,\ell,D}=\phi_{f,\ell',D'}H_p$.
\end{theorem}

The original motivation of this result was to distil some algebraic topology
from the geometric part of the proof of the main theorem of \cite{M1}.

\begin{proof} The bijection $h_{p,D,D'}$ is defined similarly to $h_{j,D}$.
The isomorphism $H_p$ is defined by assigning to any $b$-based loop $j$
the $b'$-based loop $j'$ defined as the product of the paths $\ell_p$, $j$
and $\bar\ell_p$, where $\ell_p$ is in turn the product of $\ell'$, $p$,
$\bar\ell$ and $\ell$.
\end{proof}

\begin{proposition}\label{2.4} The image of the monodromy map is transitive.
In particular, the order of $\pi(f)$ is divisible by $\Deg f$.
\end{proposition}

This follows easily from

\begin{lemma}\label{2.5} \cite{M1}*{\S2, proof of Lemma 2}
Let $f\:N\to M$ be a generic smooth map between closed oriented connected
$n$-manifolds, $n\ge 1$.
If $x,y\in N$ are such that $f(x)$ and $f(y)$ are $f$-regular values, any
$f$-transverse path joining $f(x)$ and $f(y)$ extends (with respect to
a fixed inclusion $I\emb S^1$) to an $f$-transverse loop $l\:S^1\to M$
such that $(f^*l)^{-1}(x)$ and $(f^*l)^{-1}(y)$ are singletons and
lie in the same connected component of $(l^*f)^{-1}(S^1)$.
\end{lemma}

\subsection{2-prems and string links}

\begin{proposition}\label{2.6} If $f$ is a $2$-prem, then the monodromy map
$\phi_f$ factors through the projection $T_{|\Deg(f)|}\to S_{|\Deg(f)|}$,
where $T_k$ is the group of concordance classes of string links of
multiplicity $k$.
\end{proposition}

By a {\it string link} of multiplicity $k$ we mean an embedding
$g\:(\Z/k)\x I\emb\C\x I$ sending $(\Z/k)\x\{j\}\subset\C\x\{j\}$ to itself for
$j=0,1$.
If moreover $g$ sends each $((e^{2\pi i/k})^n,j)$ to itself for $j=0,1$,
then $g$ is a {\it pure string link}.
Thus the group $C_k$ of concordance classes of pure string links of multiplicity
$k$ is the kernel of the projection $T_k\to S_k$.

\begin{proof} Let $p\:M\x\R^2\to M$ be the projection, and let us consider
an $f$-transverse loop $j\:(I,\partial I)\to (M,b)$.
Since $f$ factors through an embedding of $N$ into $M\x\R^2$, its pullback
$j^*f$ factors through an embedding of $(j^*f)^{-1}(I)$ into the pullback
$(j^*p)^{-1}(M\x\R^2)=I\x\R^2$.
This embedding is a string link.
Similarly a coherent homotopy gives rise to a concordance.
\end{proof}

Proposition \ref{2.6} is analogous to the ``only if'' implication in

\begin{theorem}[Hansen \cite{H1}, \cite{H2}]\label{2.7} 
A $d$-fold covering $f\:N\to M$ is a $2$-prem if and only if the monodromy
$\pi_1(M)\to S_d$ factors through the projection $B_d\to S_d$, where $B_d$
denotes the braid group on $d$ strands.
\end{theorem}

The ``if'' implication can be proved as follows (compare \cite{DH}*{statement of Theorem 2}).
Let $D$ be an open disk in $M$; then $\pi_1(M\but D)$ is the free group $\left<x_1,y_1,\dots,x_g,y_g\right>$,
where $g$ is the genus of $M$, and the inclusion $M\but D\emb M$ induces a homomorphism $\pi_1(M\but D)\to\pi_1(M)$
whose kernel is the normal closure of $[x_1,y_1]\dots[x_g,y_g]$.
Let $\phi\:\pi_1(M)\to B_d$ be the given homomorphism, and let $r\:M\but D\to W$ be a deformation retraction
onto a wedge of $g$ copies of $S^1$.
Then the braids $\phi(x_1),\phi(y_1),\dots,\phi(x_g),\phi(y_g)$ combine to yield the desired lift 
$\bar f_0\:f^{-1}(W)\to W\x\R^2$ of the restriction of $f$ over $W$, and its pullback $r^*(\bar f_0)$ is 
the desired lift of the restriction of $f$ over $M\but D$.
Now over $\partial D$ the latter partial lift restricts to the braid $\phi([x_1,y_1]\dots[x_g,y_g])$, which is
trivial, since $\phi$ is a homomorphism; hence the lift extends over $D$.

The ``only if'' part of \ref{2.7} along with Petersen's results discussed in \S\ref{1}
have the following group-theoretic consequence: every homomorphism $G\to S_d$,
where $G$ is a finitely generated free abelian group, factors through $B_d$.
In particular, since $B_d\to S_d$ factors through $T_d$, we get

\begin{corollary}\label{2.8} Let $f$ be a generic smooth map between compact
connected oriented $n$-manifolds, $n\ge 2$.
If the monodromy $\phi_f$ factors through a free product of finitely generated
free abelian groups then it also factors through $T_{|\Deg(f)|}$.
\end{corollary}

Another immediate thing to note about \ref{2.6} and \ref{2.7} is that every $f\:N\to M$
factors into the composition of the embedding $\Gamma_f\:N\emb N\x M$ and
the projection $N\x M\to M$.
Hence

\begin{proposition}\label{2.9}
(a) The monodromy $\pi_1(M)\to S_d$ of every $d$-fold covering $N\xr{f}M$
between surfaces factors through the group $B_d(N)$ of braids in $N\x I$.

(b) The monodromy $\pi(f)\to S_{|\Deg(f)|}$ of every generic map $N\xr{f}M$
between orientable surfaces factors through the group $T_{|\Deg(f)|}(N)$ of
concordance classes of string links in $N\x I$.
\end{proposition}

Le Dimet showed that the natural map $B_d\to T_d$ is injective, i.e.\ concordant braids are isotopic
(cf.\ \cite{Hi}*{p.\ 312}).
Indeed, the Artin representation $B_d\to\Aut(F_d)$ is injective and agrees
with the representation $T_d\to\Aut(F_d/\gamma_n)$,
which is well-defined for each $n$ by the Stallings Theorem on
the lower central series $\gamma_n$ (see \cite{HL2}*{\S1}).
But $\bigcap\ker[\Aut(F_d)\to\Aut(F_d/\gamma_n)]=1$ since
$\bigcap\gamma_n=1$ in $F_d$.

On the other hand, $B_d$ and $T_d$ have a common quotient, the homotopy braid
group $HB_d$ (see \cite{Go}, where the difference between $B_d$ and $HB_d$
is explained).
Indeed, every string link is link homotopic to a braid (see \cite{HL1}) and
concordance implies link homotopy by a well-known result of Giffen and
Goldsmith.
The latter also follows from the injectivity of $HB_d\to\Aut(F_d/\mu_0)$
\cite{HL1}, where $\mu_0$ is the product of the commutator subgroups of
the normal closures of the generators of $F_d$, which contains
$\gamma_{d+1}$.

Similarly to Artin's combing
$P_d\simeq (\dots(F_1\ltimes F_2)\ltimes\dots)\ltimes F_{d-1}$ of the pure
braid group, the kernel $HP_d$ of the projection $HB_d\to S_d$ admits
the combing $HP_d\simeq
(\dots(F_1/\mu_0\ltimes F_2/\mu_0)\ltimes\dots)\ltimes F_{d-1}/\mu_0$
\cite{Go}, \cite{HL1}.
Hence $HP_d$ is torsion-free.
Using this, Humphries proved that $HB_d$ is torsion-free for $d<7$;
in fact he showed that $\alpha\in HB_d$ has infinite order if its image in
$S_d$ has order divisible by $2$, $3$, or $5$ \cite{Hu}.

\begin{corollary}\label{2.10} Let $f$ be a generic smooth map between compact
connected oriented $n$-manifolds, $n\ge 2$.
If $\pi(f)$ contains an element $\alpha$ of finite order whose monodromy
$\phi_f(\alpha)\in S_{|\Deg(f)|}$ is of order divisible by $2$, $3$ or $5$,
then $f$ is not a $2$-prem.
\end{corollary}

Note that for the hypothesis to hold, the group $HT_{|\Deg(f)|}(N)$ of
link homotopy classes of string links in $N\x I$ must contain torsion,
by Proposition \ref{2.9}(b).

Taking into account the Yamamoto--Akhmetiev Theorem \cite{M1}, we have

\begin{corollary}\label{2.11} If $f$ is a generic smooth map from $S^2$ to a closed
orientable surface then $\pi(f)$ contains no torsion with monodromy of order
divisible by $2$, $3$ or $5$.
\end{corollary}

An interesting question is whether already some lower central series quotient
of $HB_d$ is torsion free.
For instance, the abelianization of $HB_3$ is not: the braid
$[\sigma_{12},\sigma_{23}]$ projects nontrivially to $S_3$, but it is easy to
check that the three strands of the pure braid $[\sigma_{12},\sigma_{23}]^3$
are pairwise unlinked.
This means, in particular, that it would not be a good idea to simplify
the definition of $(b,f)$-coherent homotopy into ``link map bordism'', i.e.\ to
allow the positive components of the preimage to change by pairwise
disjoint bordisms.

\begin{remark}\label{2.12}
It was shown by Habegger and Lin \cite{HL2}*{\S1} that 
\begin{enumerate}[label=(\roman*)]
\item the image
of the group $C_d$ of concordance classes of pure string links in
$\Aut(F_d/\gamma_n)$ is the subgroup $\Aut_0(F_d/\gamma_n)$ that
depends on the chosen set $\{x_1,\dots,x_n\}$ of free generators of $F_d$ and
consists of those automorphisms that send the coset $\bar x_i$ of each
$x_i$ to a conjugate of $\bar x_i$ and fix the product
$\bar x_1\cdots\bar x_d$; 
\item $\Aut_0(F_d/\gamma_2)=1$, and each
$\Aut_0(F_d/\gamma_{n+1})$ is a central extension of $\Aut_0(F_d/\gamma_n)$
by a free abelian group, which they denote $K_{n-1}$; in particular, for $n>2$,
$\Aut_0(F_d/\gamma_n)$ is torsion-free and nilpotent of class $n-2$.
\end{enumerate}

It follows easily from these that
\begin{enumerate}[label=(\roman*$'$)]
\item the image of $T_d$ in
$\Aut(F_d/\gamma_n)$ is the subgroup $\Aut_1(F_d/\gamma_n)$ consisting of those
automorphisms that send the coset $\bar x_i$ of each $x_i$ to a conjugate of
some $\bar x_j$ and fix the product $\bar x_1\cdots\bar x_d$;
\item $\Aut_1(F_d/\gamma_2)\simeq S_d$, and for $n>1$, each
$\Aut_1(F_d/\gamma_{n+1})$ is a central extension of $\Aut_1(F_d/\gamma_n)$ by
the same free abelian group $K_{n-1}$.
\end{enumerate}

Thus the homomorphism $T_d\to S_d$ factors through the limit of the inverse
sequence $\dots\to\Aut_1(F_d/\gamma_3)\to\Aut_1(F_d/\gamma_2)\simeq S_d$.
If some term of this inverse sequence or the inverse limit is torsion-free
(for each $d$) --- or if $HB_d$ is torsion-free for all $d$ --- then
the restriction on the order of the monodromy is superfluous in
Corollary \ref{2.10}.
\end{remark}

\section{Some computations of $\pi(f)$}\label{3}

\subsection{A certain fold map $S^n\to S^n$ of degree $d$}

Let $f$ be the degree $d$ map $f\:S^n\to S^n$ defined by picking $d+1$ disjoint
$n$-disks in $S^n$ and sending each of them homeomorphically to its own exterior in $S^n$.
Let $b$ be a point with $|f^{-1}(b)|=d$ (i.e.\ a point in the interior of
one of the disks).

It is easy to see that $f$ lifts to an embedding $f\x g\:S^n\emb S^n\x\R^2$.
Namely, fix an embedding of $pt*[d+1]$ (the cone over $[d+1]=\{0,\dots,d\}$)
into $\R^2$, and let $g\:S^n\to pt*[d+1]\subset\R^2$ send the interior of the $i$th
$n$-disks into $pt*\{i\}\but pt*\emptyset$, and the exterior of the disks
into $pt*\emptyset$.

\begin{example}[the case $d=2$]\label{3.1.1}
Let $\alpha\in\pi(f,b)$ be the class of a loop intersecting each disk along its
diameter (compare Example 6 in \cite{M1}).
Then $\alpha$ is nontrivial since it is easily seen to have a nontrivial
monodromy $\phi_f(\alpha)\in S_2$.
On the other hand, by Proposition \ref{2.6} the monodromy map $\phi_f$ lifts to
$\hat\phi_f\:\pi(f,b)\to T_2$.
Since $T_2\to S_2$ factors through $HB_2$, we conclude that
the image of $\hat\phi_f(\alpha)$ in $HB_2\simeq\Z$ is nontrivial.
Hence $\alpha$ has infinite order, and the composition
$\Z\simeq\left<\alpha\right>\subset\pi(f,b)\to HB_2\simeq\Z$ is an isomorphism.
Thus $\pi(f)$ contains a direct summand isomorphic to $\Z$.
\end{example}

\begin{example}[the case $d=2$, $n>2$]\label{3.1.1.1}
We will now show that if $n>2$ (and still $d=2$), then $\pi(f)$ is isomorphic
to $\Z$.

Let $A$, $B$ and $C$ denote the $3$ disks, with $b\in B$.
A loop representing an element of $\pi(f,b)$ gives rise to a word in
the alphabet $\{A,B,C\}$ (starting and ending with the letter $B$), which
encodes the sequence of disks intersected by the loop.
If two loops give rise to the same word, then (using that $n\ge 3$)
they represent the same element of $\pi(f,b)$.
Furthermore, it is easy to see%
\footnote{The author is grateful to P. M. Akhmetiev for pointing out these
relations.} that $XX=X$ and $XYX=X$ in $\pi(f,b)$ for any $X,Y\in\{A,B,C\}$.

Let $F$ be the free monoid (=semi-group with $1$) on the alphabet $A,B,C$
(where the product of words is given by concatenation), and let $BFB$ be
the submonoid of $F$ consisting of all words of the form $BwB$, where
$w\in F$.
Let $G$ be the quotient of $BFB$ by the relations $XYX=X=XX$, where
$X,Y\in\{A,B,C\}$.
Then $G$ is a group with unit $B$ and with the inverse given by
$BX_1\dots X_nB\mapsto BX_n\dots X_1B$.
Indeed, $BX_n\dots X_1BBX_1\dots X_nB=BX_n\dots X_1BX_1\dots X_nB=
BX_n\dots X_1\dots X_nB=\dots=B$.

In fact, $G$ is nothing but the group of simplicial loops in the triangle
$\partial\Delta^2$ (with vertices $A$, $B$, $C$) under the relation of simplicial homotopy.
Thus $G\simeq\Z$, with $n\in\Z$ corresponding to the class of
$B(ABC)^nB$ (where $(ABC)^{-1}=CBA$).
By construction, we have an epimorphism $G\to\pi(f,b)$ sending a generator
onto $\alpha$.
Hence $\pi(f)\simeq\Z$.
\end{example}

\begin{example}[$n=2$, $d=2$]\label{3.1.1.2}
In the case $n=2$, $d=2$, $\pi(f)$ is larger than $\Z$, for similarly to
Example \ref{3.2.2} below it can be shown that there are loops giving rise to
the words $B$, $BAB$ and $BCB$ yet representing elements of infinite order
in $\pi(f)/\Z$.
\end{example}

\begin{example}[$n>2$, $d>2$]\label{3.1.2}
In the case $d>2$, $n>2$, the same considerations as in \ref{3.1.1.1} show that
$\pi(f)$ is a quotient of $\pi_1((\Delta^{d+1})^{(1)})$ (that is, of the free group
on $\frac{d(d-1)}2$ letters), and admits an epimorphism onto the homotopy braid
group $HB_d$.
\end{example}

\subsection{Fold maps of geometric degree $1$ with embedded spherical folds}\label{3.2}
Let $f\:S^n\to S^n$ be a generic fold map that embeds its fold surface
$\Sigma_f:=\{x\in S^n\mid\ker df_x\ne 0\}$ and is such that $|f^{-1}(b)|=1$;
in particular, $f$ has degree $\pm 1$.
({\it Fold map} means that every point of $\Sigma_f$ is a fold point, rather than
a point of a higher singularity type.)
In the case $n>2$, let us additionally assume that each component $S_i$ of
$\Sigma_f$ is a sphere.

The following notation will be used.
If $S_i$ is a component of $\Sigma_f$, let $B_i$ be the $n$-ball bounded
by $S_i$ in $S^n\but f^{-1}(b)$.
On the other hand, let $D_i$ be the $n$-ball bounded by $f(S_i)$ in
$S^n\but b$.
Each $f(B_i)$ is connected, hence coincides with some $D_{\rho(i)}$, таким что $D_i\subset D_{\rho(i)}$.
Note that $\rho\rho=\rho$; in other words if $j$ is of the form $\rho(i)$,
then $f(B_j)=D_j$.
In particular, in this case $S_j$ is {\it outer}, i.e.\ $f$ sends
a neighborhood of $S_j$ in $B_j$ into $D_j$.
Accordingly, we call an $S_i$ {\it inner} if $f$ sends such a neighborhood into
the closure of $S^2\but D_i$.

\begin{example}[a two-dimensional example]\label{3.2.1} There exists a generic fold map
$f\:S^2\to S^2$ such that $f$ embeds $\Sigma_f$ and $|f^{-1}(b)|=1$, yet
$\pi(f,b)$ is non-trivial.

Namely, $f$ is the unique (up to reparametrization) map such that $\Sigma_f$ is
the union of four curves $S_1$, $S_2$, $S_3=S_{\rho(1)}$ and $S_4=S_{\rho(2)}$
such that $D_1\cap D_2=\emptyset$ and $D_1\cup D_2\subset D_3\subset D_4$.

Let $l\:(S^1,pt)\to(S^2,b)$ be a loop intersecting $D_3$ by a diameter that
separates $D_1$ from $D_2$ within $D_3$, and intersecting $D_4$ by some diameter (which contains
the former diameter).
Then the pullback $(l^*f)^{-1}(S^1)$ consists of three components $P_0$, $P_1$ and $P_2$,
with $P_0$ containing $(l^*f)^{-1}(pt)$ and with $(f^*l)(P_1)$ and $(f^*l)(P_2)$
contained respectively in $B_1$ and $B_2$.
By a homotopy $h$ from $l$ to an $\tl l$ with values in $S^2\but (D_1\cup D_2)$
one cannot eliminate either $P_1$ or $P_2$; that is, neither $P_1$ nor $P_2$
bounds a disk in $(h^*f)^{-1}(S^1\x I)$.
On the other hand, any such homotopy $h$ with values not only in
$S^2\but (D_1\cup D_2)$ joins either $P_1$ or $P_2$ to $P_0$;
that is, either $P_1$ or $P_2$ (or both) belongs to the component $Q$ of
$(h^*f)^{-1}(S^1\x I)$ containing $P_0$.
Hence $Q$ has at least two boundary components in $(l^*f)^{-1}(S^1)$, and
so $h$ is not coherent.
Thus one cannot eliminate either $P_1$ or $P_2$ by a coherent homotopy.
So there exists no coherent null-homotopy of $l$.

Similar considerations show that no power of $l$ is trivial in $\pi(f,b)$.
\end{example}

\begin{example}[a mild generalization]\label{3.2.2}
Let us generalize Example \ref{3.2.1} to show that if (under the hypothesis of \S\ref{3.2})
there exists a pair $(i,j)$ such that $D_i\subset f(B_j)$ and $D_j\subset f(B_i)$
(for brevity, we shall call such a pair $(i,j)$ {\it linked}) and additionally
$B_1\cap B_2=\emptyset$, then $\pi(f,b)$ is nontrivial as long as $n=2$.

Indeed, up to renumbering we may assume that $(i,j)=(1,2)$.
Since $(1,2)$ is linked, $\rho(1)\ne 1$.
Then, in particular, $S_{\rho(1)}$ is outer.
Hence if $S_1$ is also outer, then there exists an $i$ such that
$B_1\supset B_i\supset B_{\rho(1)}$ and $S_i$ is inner.
Then $\rho(i)=\rho(1)$ and $(i,2)$ is a linked pair.
Thus without loss of generality we may assume that $S_1$ is inner; similarly for $S_2$.

The construction of Example \ref{3.2.1} will apply here once we show that there is
a loop $l\:(S^1,pt)\to (S^2\but (D_1\cup D_2),b)$ such that the component
$L^x$ of $L:=(l^*f)^{-1}(S^1)$ that contains $(l^*f)^{-1}(pt)$ represents,
via $f^*l$, a nontrivial element of $H_1(S^2\but (B_1\cup B_2))$.
Let $B$ be the union of all disks bounded by the collection of circles
$f^{-1}(f(S_1\cup S_2))$ in the complement to $x:=f^{-1}(b)$.
Let $S_1^+$ be a pushoff of $S_1$ into the complement of $B_1$, and let
$y$ be a point of $S_1^+$.
Since $S_1$ is inner, $y\notin B$; and since $S^2\but B$ is connected,
$x$ and $y$ can be joined by a path $p$ in $S^2\but B$.
By Lemma \ref{2.5} there exists a loop $l_0\:(S^1,pt)\to(S^2,b)$ such that
$x_0:=(l_0^*f)^{-1}(pt)=(f^*l_0)^{-1}(x)$ and $y_0:=(f^*l_0)^{-1}(y)$ are
singletons and lie in the same component $L^x_0$ of $L_0:=(l_0^*f)^{-1}(S^1)$.
Moreover, by the proof of Lemma \ref{2.5} (found in \cite{M1}) we may
assume that $l_0$ has values in a neighborhood of $f(p(I))$, hence in
$S^2\but (D_1\cup D_2)$.
Then $f^*l_0$ sends $L^x_0$ into $S^2\but (B_1\cup B_2)$.
Now amend $l_0$ by cutting it open at $f(y)$ and inserting a loop circling
around $f(S_1^+)$.
Then the resulting loop $l_1\:(S^1,pt)\to (S^2\but (D_1\cup D_2),b)$ is
such that the component $L_1^x$ of $L_1:=(l_1^*f)^{-1}(S^1)$ that contains
$x_1:=(l_0^*f)^{-1}(pt)$ differs from $L_0^x$ by a loop circling around $S^1_+$.
Thus $f^*l_0|_{L_0^x}$ and $f^*l_1|_{L_1^x}$ represent distinct elements
of $H_1(S^2\but (B_1\cup B_2))$; so at least one of them is non-trivial.
\end{example}

\begin{example}[a higher-dimensional proposition]\label{3.2.3} We shall show that
the phenomenon exhibited in Example \ref{3.2.1} does not occur in higher dimensions;
more specifically, that (under the hypothesis of \S\ref{3.2}) for $n>2$, a $b$-based
$f$-transverse loop crossing each $S_i$ at most twice represents the trivial
element of $\pi(f,b)$.

Indeed, let $l\:(S^1,pt)\to (S^n,b)$ be such a loop.
Since $n>2$, it may be assumed to be embedded.
(This assumption will not be essentially used, but allows to simplify notation.)
Since $l$ is $f$-transverse, it is transverse to the codimension one submanifold
$f(\Sigma_f)$.
Up to a renumbering of $S_i$'s, we may assume that $D_1$ meets $l(S^1)$ and
is innermost among all the balls $D_i$ in $S^n\but b$ that meet $l(S^1)$.
Write $A=l(S^1)\cap D_1$; by our hypothesis, it is an arc.
Then $\partial A$ bounds an arc $A'$ in $f(S_1)$ so that $A'$ meets $l(S^1)$
only in $\partial A'=\partial A$.
The circle $A\cup A'$ bounds a $2$-disk $D$ in $D_1$, meeting $\partial D_1$
only in $A'$.
Without loss of generality $D$ meets $l(S^1)$ only in $A$.
Since $n>2$, we may assume that $D$ is disjoint from all $D_j$ that lie in
the interior of $D_1$.
Let us homotop $l$ across $D$, from $A$ to $A'$, to a loop $l_1$ such that
$l_1(S^1)\cap D_1$ has fewer components than $l(S^1)\cap D_1$.
Proceeding inductively, we obtain a pointed homotopy $h_t$, $t\in [0,\dots,N]$,
from $l_0:=l$, through loops $l_1,l_2,\dots$, to a loop $l_N$ disjoint from
every $D_i$.
Since $n>2$, the latter is pointed null-homotopic with values in
the complement to all $D_i$'s and so represents the trivial element of $\pi(f)$.

Let us show that the constructed null-homotopy of $l$ is coherent.
Let $L_i=(l_i^*f)^{-1}(S^1)$.
If $S_{i+1}$ is outer, then $L_{i+1}$ is obtained from $L_i$ by removing
one component, disjoint from $(l_i^*f)^{-1}(pt)$.
Else (i.e.\ if $S_{i+1}$ is inner) $L_{i+1}$ is obtained from $L_i$ by
splitting one of the components into two.
(Here we are using our hypothesis, implying by induction that $l_i$ crosses
$S_{i+1}$ just twice.)
The component of $L_i$ being split may contain $(l_i^*f)^{-1}(pt)$, in which
case of the two resulting components $P$, $Q$ of $L_{i+1}$ one (say, $P$) would
contain $(l_{i+1}^*f)^{-1}(pt)$.
Then for the constructed null-homotopy to be coherent, the other component $Q$
has to be glued up by a disk in $(h_{[i+1,N]}^*f)^{-1}(S^1\x I)$, where
$h_{[i,j]}\:S^1\x I\to N$ denotes the interval from $l_i$ to $l_j$ in
the constructed null-homotopy $h_t$.
Indeed, by construction, $Q$ will be glued up by a disk already in
$(h_{[i+1,\rho(i)]}^*f)^{-1}(S^1\x I)$.
\end{example}

\begin{example}[a two-dimensional proposition]\label{3.2.5}
Let us show that the proposition in Example \ref{3.2.3} remains valid in dimension two under
the additional hypothesis that there are no linked pairs $(i,j)$.

Indeed, let us examine the argument of \ref{3.2.3}.
It contains only two essential applications of the condition $n>2$:
to conclude that $D$ is disjoint from all $D_j$'s that lie in the interior
of $D_1$ and to coherently null-homotop $l_N$.
Now if $n=2$ and $D$ meets some $D_j$, then it follows from our assumption
of $(1,j)$ being unlinked that $D$ contains $f(B_j)$.
(For if it doesn't contain, then taking into account our assumption that $D_1$
is innermost among all $D_i$'s meeting $l(S^1)$, the only possibility is that
$f(B_j)$ contains $D_1$.
However $D_j\subset D_1\subset f(B_1)$, so $(1,j)$ is linked.)
Thus $D$ contains each $D_j$ together with its $f(B_j)$, which implies
that $f^{-1}(D)$ is homeomorphic to a disjoint union of disks, each
bounded by a component of $f^{-1}(\partial D)$.
Thus the homotopy of $l$ along $D$ will be coherent.
Similarly the null-homotopy of $l_N$ will be coherent.
\end{example}

\begin{center}
\smallskip

\includegraphics{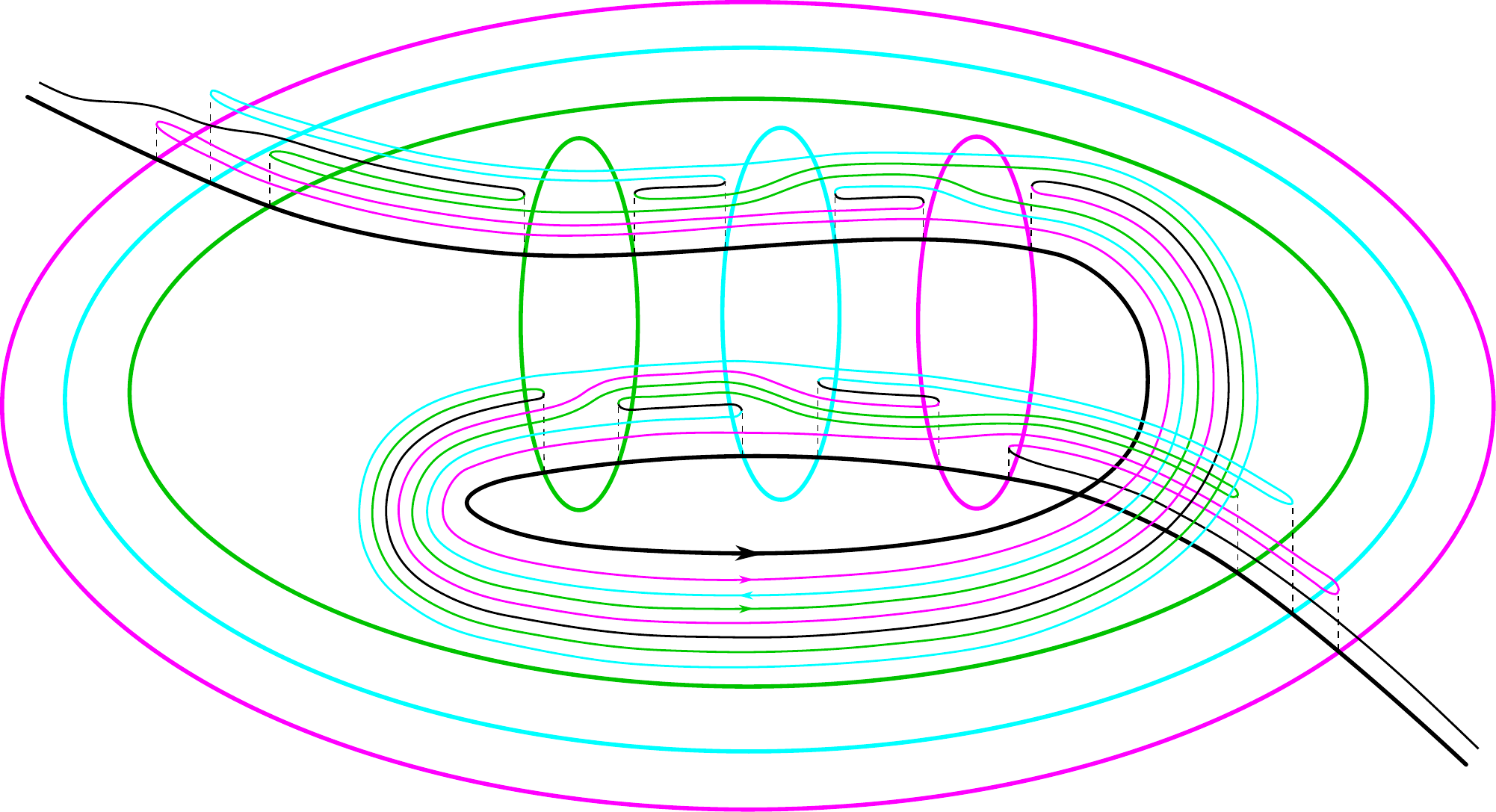}
\smallskip

{\bf Fig.\ 1.} A map $f\:S^n\to S^n$ of geometric degree $1$ with nontrivial $\pi(f)$.
\end{center}
\bigskip

\begin{example}[a higher-dimensional example]\label{3.2.4}
In fact, $\pi(f,b)$ need not be trivial for $n>2$ (under the hypothesis of \S\ref{3.2})
as shown by the map depicted in Figure 1.
The six thick colored (or grayscale, depending on the reader's medium) ellipses depict the folds; 
the thick black curve is the loop $l$ in question; and the thin curves illustrate the pullback $l^*f$.
The arrows mark those of the two components of $(l^*f)^{-1}(S^1)$ that passes through $(l^*f)^{-1}(pt)$.

The curve $l(S^1)$ meets each of the three $n$-balls $D_i$ bounded by the images of the three inner spheres
of folds $S_i$, $i=1,2,3$, in two arcs $J_i$, $J_i'$.
A key feature of this picture is that the marked component has a fold over one endpoint of each of 
the six arcs, and the other (unmarked) component has folds over their opposite endpoints.
Thus if any of the six arcs is eliminated as in Example \ref{3.2.3}, this would result in the marked
component being joined to the unmarked one, whence the eliminating homotopy would fail to be coherent
(cf.\ Example \ref{3.2.1}).

Moreover, no preliminary tampering with the six arcs by a coherent homotopy of $l$ is going to help.
Indeed, if $H\:D^2\to S^n$ is a coherent 
null-homotopy of $l$, then the pullback $H^*f$ of $f$ has a fold curve over each $H^{-1}(S_i)$.
Each $H^{-1}(D_i)$ is a codimension zero
submanifold in $D^2$, whose boundary contains the two arcs $l^{-1}(J_i)$ and $l^{-1}(J_i')$.
If $j\ne i$, then these arcs alternate with $l^{-1}(J_j)$ and $l^{-1}(J_j')$ with respect to the cyclic
order on $S^1$, whereas $H^{-1}(D_i)$ and $H^{-1}(D_j)$ are disjoint.
Hence either $l^{-1}(J_i)$ and $l^{-1}(J_i')$ are contained in different components of $H^{-1}(D_i)$,
or $l^{-1}(J_j)$ and $l^{-1}(J_j')$ are contained in different components of $H^{-1}(D_j)$ (or both assertions
hold).
By symmetry, we may assume the former.
Since the component of $H^{-1}(D_i)$ containing $l^{-1}(J_i)$ does not contain $l^{-1}(J_i')$, its boundary
component containing the arc $l^{-1}(J_i)$ otherwise contains only points of $H^{-1}(S_i)$, which therefore
must constitute an arc. 
Thus the two endpoints of the arc $l^{-1}(J_i)$ belong to the same component of $H^{-1}(S_i)$.
Hence the fold curve of $H^*f$ over this component constitutes a path in $(H^*f)^{-1}(D^2)$ starting on 
the marked component of $(l^*f)^{-1}(S^1)$ and ending on the other (unmarked) component.
Therefore the latter two components are joined into one in the null-homotopy, which is therefore non-coherent.
Thus $l$ is not coherently null-homotopic.
\end{example}

\end{document}